\documentclass[11pt,a4paper]{article}

\usepackage[utf8]{inputenc}
\usepackage{url,a4wide}
\usepackage{graphicx} % Bilder einbinden
\usepackage{color} % Farben
\usepackage{enumerate,paralist} % Aufzaehlungen ,paralist:  
\usepackage{amsmath,amsfonts,amssymb,amsthm} % Mathemodus
\usepackage{ifthen}
\usepackage{tikz}
\usetikzlibrary{calc,through}

\theoremstyle{plain}
\newtheorem{theorem}{Theorem} 

\newtheorem{corollary}[theorem]{Corollary}
\newtheorem{lemma}[theorem]{Lemma}
\newtheorem{proposition}[theorem]{Proposition}
\theoremstyle{definition}
\newtheorem{definition}[theorem]{Definition}
\newtheorem{example}[theorem]{Example}

\newcommand{\dotcup}{\ensuremath{\mathaccent\cdot\cup}}

\newcommand\Z{{\mathbb Z}}

\newcommand\R{{\mathbb R}}

\newcommand\conv{\mathop{\rm conv}}

\newcommand\aff{\mathop{\rm aff}}
\newcommand\relint{\mathop{\rm relint}}

\begin{document}

\title{Inscribable stacked polytopes}

\author{Bernd Gonska \\
FU Berlin, Inst.\ Mathematics\\
Arnimallee 2, 14195 Berlin, Germany\\
\url{gonska@math.fu-berlin.de}
\and 
\setcounter{footnote}{0}G\"unter M. Ziegler%
\thanks{Supported the European Research
Council under the European Union's Seventh Framework Programme (FP7/2007-2013)/\allowbreak
ERC Grant agreement no.~247029-SDModels and by the DFG Research Center
\textsc{Matheon} “Mathematics for Key Technologies” in Berlin} \\
FU Berlin, Inst.\ Mathematics\\
Arnimallee 2, 14195 Berlin, Germany\\
\url{ziegler@math.fu-berlin.de}}

\date{November 22, 2011} 
\maketitle

\begin{abstract}
 We characterize the combinatorial types of stacked $d$-polytopes that are inscribable.
 Equivalently, we identify the triangulations of a simplex by stellar subdivisions 
 that can be realized as Delaunay triangulations.
\end{abstract}

\section{Introduction}  

\begin{quote}
\emph{How do geometric constraints restrict the combinatorics of polytopes?}
\end{quote}
One instance of this question asks for the combinatorial types of $d$-polytopes
that are \emph{inscribable}, that is, that have realizations with all vertices
on a sphere. This question was raised in 1832 by 
Jacob Steiner \cite{Steiner}, who asked whether all $3$-dimensional polytopes are inscribable.
The answer was given by Ernst Steinitz \cite{Steinitz} nearly 100 years later:
No --- despite a claim to the contrary by Brückner \cite[p.~163]{Bruck} this is not even
true for all simplicial types.
Indeed, the polytope obtained by “stacking” a new vertex onto each facet
of a tetrahedron is not inscribable (compare Grünbaum \cite[Sect.~13.5]{Gruenbaum}).

General $3$-polytopes are inscribable if any only if a certain associated linear program is feasible:
This was proved by Hodgson, Rivin \& Smith \cite{Rivin} using hyperbolic geometry; a complete combinatorial characterization 
for inscribability of $3$-polytopes is not available up to now (compare Dillencourt \& Smith~\cite{Dill96}).

In this paper we embark on a study of inscribability for $d$-dimensional convex polytopes.
One instance our motivating question is 
\begin{quote}
	\emph{How does the condition of inscribability restrict the $f$-vectors of polytopes?}
\end{quote} 
In this context, we observe (Section~\ref{subsec:f-inscribable_3poly}) that all $f$-vectors of $3$-polytopes
(as characterized by Steinitz \cite{Steinitz} in 1906) also occur for inscribable polytopes.
Also, as the cyclic polytopes are inscribable, 
we get that the Upper Bound Theorem of McMullen~\cite{McM} 
is sharp for the restricted class of incribable polytopes (Section~\ref{subsec:f-inscribable_cyclic}).
Moreover, it is easy to see that there are stacked $d$-polytopes with $d+1+n$ vertices
that are inscribable ($d\ge2$, $n\ge0$), so
we find that also the Lower Bound Theorem of Barnette \cite{Barnette1} \cite{Barnette2}
is sharp for the restricted class of inscribable polytopes. 
One is thus naturally led to ask whether all $f$-vectors of convex polytopes,
or at least all $f$-vectors of simplicial convex polytopes can be obtained from 
inscribable polytopes. This will be further studied in~\cite{GonskaZiegler2}.

One can then proceed and try to characterize inscribability for some of
these classes. 
This seems out of reach for neighborly polytopes, as according to Shemer \cite{Shemer} there are huge numbers of
combinatorial types, and no combinatorial 
classification in sight.
However, as the main result of this paper we
provide a combinatorial characterization of inscribable stacked polytopes.
It refers to the dual tree of a stacked polytope, which will be formally defined in Section~\ref{subsec:stacked} below;
for now, we refer to Figure~\ref{Pic:DualTree}.

\begin{figure}[htb!] 
	\centering
    \begin{tikzpicture}[
rotate around={60:(1,1)}, 
x  = {(-1,0)},y  = {(0cm,0.5cm)},z  = {(0cm,.7cm)}, scale=1.5,rounded corners=.66pt]
	\coordinate (P1) at (0,0,5.5); 
	\coordinate (P2) at (-.5,1,4);
	\coordinate (P3) at (1.5,0,5);
	\coordinate (P4) at (2,0,4);
	\coordinate (P5) at (1,1,2);
	\coordinate (P6) at (-1,0,2);
	\coordinate (P7) at (1,-1,2);
	\coordinate (P8) at (0,0,1);
	\coordinate (V1) at ($.25*(P1)+.25*(P2)+.25*(P5)+.25*(P6)$);
	\coordinate (V2) at ($.25*(P1)+.25*(P5)+.25*(P6)+.25*(P7)$);
	\coordinate (V3) at ($.25*(P1)+.25*(P3)+.25*(P5)+.25*(P7)$);
	\coordinate (V4) at ($.25*(P3)+.25*(P4)+.25*(P5)+.25*(P7)$);
	\coordinate (V5) at ($.25*(P5)+.25*(P6)+.25*(P7)+.25*(P8)$);
	% lexikographisch: -----------------------------
	\draw [thin] (P1) -- (P2) -- (P5) -- cycle;
	\draw [thin] (P1) -- (P3) -- (P5) -- cycle;
	\draw [thin] (P2) -- (P5) -- (P6) -- cycle;
	\draw [thin] (P3) -- (P4) -- (P5) -- cycle;
	\draw [thin] (P4) -- (P5) -- (P7) -- cycle;
	\draw [thin] (P5) -- (P6) -- (P8) -- cycle;
	\draw [thin] (P5) -- (P7) -- (P8) -- cycle;
	\draw  (P1) -- (P5) -- (P6) -- cycle;% front 
	\draw [red,ultra thick] (V2)--(V5)--(V2)--(V3)--(V4);
	\draw [blue,ultra thick] (V1)--(V2);
	\draw [fill=red,draw=red,ultra thick]
%		(V2)  circle(2pt) node [anchor=north east]{$r$};
		(V2)  circle(2pt) node [anchor=south west]{$r$};
		\draw [fill=white,draw=red,ultra thick]
		(V3)  circle(2pt)
		(V4)  circle(2pt)
		(V5)  circle(2pt);
	\draw [fill=white,draw=blue,ultra thick]
%		(V1)  circle(2pt) node [anchor=north east]{$u$};
		(V1)  circle(2pt) node [anchor=west]{$u$};
	\draw [very thick] (P1) -- (P2) -- (P6) -- cycle;% front
	\draw [very thick] (P1) -- (P3) -- (P7) -- cycle;% front
	\draw [very thick] (P1) -- (P6) -- (P7) -- cycle;% front
	\draw [very thick] (P3) -- (P4) -- (P7) -- cycle;% front
	\draw [very thick] (P6) -- (P7) -- (P8) -- cycle;% front
    \end{tikzpicture}
	\caption{The dual tree of a stacked $3$-polytope.}
	\label{Pic:DualTree}
\end{figure}
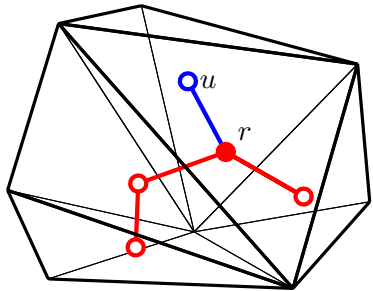

\begin{theorem}\label{mainthm:polytopes}
	A stacked polytope is inscribable if and only if all nodes of its dual tree have degree at most~$3$.
\end{theorem}

Thus the requirement of inscribability does not restrict the possible $f$-vectors of stacked polytopes.
However, other combinatorial parameters are restricted. For example, 
in any inscribable stacked $d$-polytope (other than a simplex, $d\ge3$) 
less than half of the vertices are simple,
while for general stacked $d$-polytopes roughly $\frac{d-1}d$ of the vertices can be simple. 
\smallskip

The study of inscribable convex $d$-polytopes is, via stereographic projection, 
equivalent to the study of $(d-1)$-dimensional Delaunay triangulations. 
(The importance of stereographic projection in this context was
stressed in 1979 by Brown~\cite{Brown}.) 
Under this correspondence (which is detailed in Section~\ref{subsec:stereographic}),
the stacked $d$-polytopes with $d+1+n$ vertices correspond to
the Delaunay triangulations of a $(d-1)$-simplex generated by a sequence of $n$ 
stellar subdivisions of $(d-1)$-faces.
(The rooted tree of a multiple stellar subdivision of a simplex 
is discussed in Section~\ref{subsec:Delaunay} below;
for now, we refer to Figure~\ref{Pic:DualTree2}.)

\begin{figure}[htb!] 
	\centering
    \begin{tikzpicture}[x  = {(1cm,0cm)},y  = {(0cm,1cm)}, scale=1.3,rounded corners=.66pt]
	\coordinate (P1) at (0,0); 
	\coordinate (P2) at (5,0);
	\coordinate (P3) at (2.5,3.75);
	\coordinate (V0) at ($.33*(P1)+.33*(P2)+.33*(P3)$); 
	\coordinate (V2) at ($.33*(V0)+.36*(P1)+.27*(P3)$);
	\coordinate (V21) at ($.33*(V0)+.36*(V2)+.27*(P3)$);
	\coordinate (V3) at ($.3*(V0)+.4*(P3)+.3*(P2)$);
	% lexikographisch: -----------------------------
	\draw [ultra thick] (P1) -- (P2) -- (P3) -- cycle;
	\draw [thick] (P1) -- (V0) -- (P2);
	\draw [thick] (P1) -- (V0) -- (P3);
	\draw [thick] (P2) -- (V0) -- (P3);
	\draw [thin] (P3) -- (V3) -- (V0);
	\draw [thin] (P2) -- (V3) -- (V0);
	\draw [thin] (P3) -- (V3) -- (P2);
	\draw [thin] (P1) -- (V2) -- (V0);
	\draw [thin] (P3) -- (V2) -- (V0);
	\draw [thin] (P1) -- (V2) -- (P3);
	\draw [thin] (V2) -- (V21) -- (V0);
	\draw [thin] (P3) -- (V21) -- (V0);
	\draw [thin] (V2) -- (V21) -- (P3);
	\draw [ultra thick,red] (V3) -- (V0) -- (V2) -- (V21);
	\draw [fill=red,draw=red,ultra thick]
		(V0)  circle(2pt) 
		+(0,-0.04) node [anchor=north] {$r$};
	\draw [fill=white,draw=red,ultra thick]
		(V3)  circle(2pt)
		(V2)  circle(2pt)
		(V21)  circle(2pt);
		\coordinate (TU) at (8,3.5); 
		\coordinate (T0) at (8,2.5); 
		\coordinate (T3) at (7,1.5);
		\coordinate (T2) at (9,1.5);
		\coordinate (T21) at (9,.5); 
	\draw [ultra thick,blue, dotted] (TU) -- (T0);
	\draw [ultra thick,red] (T3) -- (T0) -- (T2) -- (T21);
		\draw [fill=blue,draw=blue,ultra thick]
			(TU)  circle(2pt) node [anchor=south west] {$u$};
		\draw [fill=red,draw=red,ultra thick]
			(T0)  circle(2pt) node [anchor=south west] {$r$};
		\draw [fill=white,draw=red,ultra thick]
			(T3)  circle(2pt)
			(T2)  circle(2pt)
			(T21)  circle(2pt);
    \end{tikzpicture}
	\caption{A stellar subdivision of a simplex and its dual rooted tree.}
	\label{Pic:DualTree2}
\end{figure}
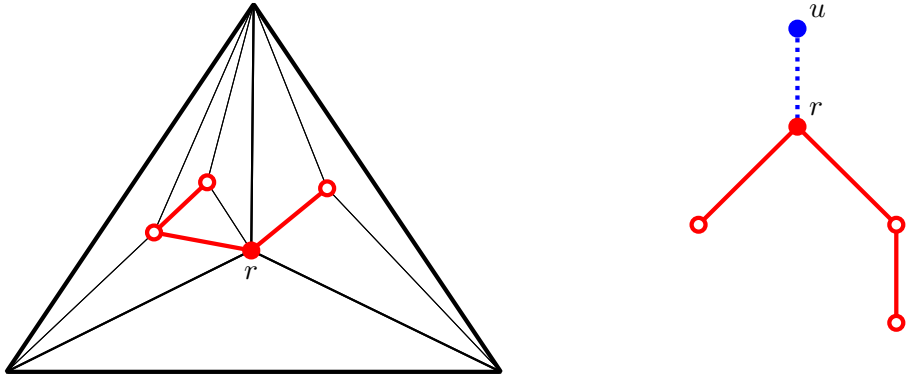

\begin{theorem}\label{mainthm:Delaunay} 
	A triangulation that is a multiple stellar subdivision of a $(d-1)$-simplex 
	can be realized as a Delaunay triangulation if and only if
	at most two of the $(d-1)$-simplices generated in any single stellar subdivision are further subdivided. 
\end{theorem}

\section{Stacked polytopes and Delaunay triangulations}\label{sec:}

\subsection{Inscribable polytopes}\label{subsec:inscribable}

\begin{definition}[inscribed polytope]
	A convex $d$-polytope is \emph{inscribed} if its vertices lie on a $(d-1)$-sphere.
	It is \emph{inscribable} if it is combinatorially equivalent to an inscribed
	polytope, that is, if it has a realization that is inscribed.
\end{definition}

\subsection{Stacked polytopes}\label{subsec:stacked}
 
\begin{definition}[stacked polytope]\label{def:stacked}
	A polytope is \emph{stacked} if it can be built from a $d$-simplex by a sequence of {stacking operations}:
	A \emph{stacking operation} is performed onto a facet
	by taking the convex hull of the polytope with a new point that lies beyond the selected facet but beneath 
	all other facets of the polytope.
\end{definition} 
  
  A stacking operation can also be imagined as gluing a simplex onto a facet.
  A simplicial $d$-polytope $P$ is stacked if and only if it has a triangulation  
  with only interior faces of dimension $d$ and $d-1$. In dimension at least $3$ such a triangulation 
  is unique if it exists: Its simplices are given by the cliques (complete subgraphs)
  of the graph of~$P$. 

The “claim to fame” of stacked polytopes is the Lower Bound Theorem \cite{Barnette1} \cite{Barnette2} \cite{Bronsted}: Among all simplicial
$d$-polytopes with $d+1+n$ vertices, the stacked polytopes have the minimal number of facets
(and indeed, the minimal number of $k$-faces, for all $k$).
Moreover, for $d\ge4$ the stacked polytopes are the only polytopes with these parameters.

\begin{definition}[dual tree of a stacked polytope]\label{def:dualtree_polytope}
	For $d\ge3$, the \emph{dual tree} $T_P$ of a stacked $d$-polytope $P$ is the dual graph of  
	its triangulation that has only interior $d$- and $(d-1)$-faces:
	Every $d$-face in the triangulation corresponds to a node and every interior $(d-1)$-face
	corresponds to an edge of the tree.
\end{definition}

The graph $T_P$ given by Definition~\ref{def:dualtree_polytope} 
is indeed a tree if $P$ is stacked.

We choose any node of $T_P$ as a root and assign an order 
to the rest of the nodes such that a child is always greater than its parent.
Any such order implies an iterative construction of $P$ via stackings in the following way:
The root represents the initial simplex. Every child has one vertex that it does not share
with its parent. This is used to stack the $(d-1)$-face that child and parent share.
Assuming that $T_P$ has at least two nodes, we see that the leaves of the tree are responsible
for the simple vertices of $P$, but no further simple vertices are possible,
except if the root has exactly one child, then there is an additional simple vertex 
that only the root $d$-simplex contains.

Clearly the dual tree $T_P$ of a stacked $d$-polytope on $d+1+n$ vertices has maximal degree
\[
\max\deg T_P\le \min\{d+1,n\}
\]
and stacked polytopes with these parameters exist for all $d\ge3$ and $n\ge0$.

\subsection{Delaunay triangulations}\label{subsec:Delaunay}

For any affinely spanning finite set of points $V\subset\R^{d-1}$, the 
\emph{Delaunay subdivision} $\mathcal D(V)$
is the unique subdivision of $\conv(V)$ into inscribed $(d-1)$-polytopes 
$P_i=\conv(V_i)$, $V_i\subset V$, that are given by the \emph{empty circumsphere condition}:
There exists a $(d-2)$-sphere that passes through all vertices of $P_i$
but all other vertices of $D(V)$ lie outside this sphere.
If the points in $V$ are in sufficiently general position (which is satisfied in particular
if no $d+1$ points lie on a sphere), then the Delaunay subdivision is indeed a triangulation,
known as \emph{the Delaunay triangulation} of~$V$.

One way to construct the Delaunay subdivision/triangulation is to 
derive it from an inscribed $d$-polytope by using the (inverse) stereographic 
projection, as discussed below.

We will employ the following very elegant criterion in order to test whether a given triangulation is 
the Delaunay triangulation. For this we call a face of a triangulation \emph{Delaunay}
if there exists a \emph{supporting sphere} of the face, that is,
a $(d-2)$-sphere passing through the vertices of the face such that all 
other vertices of the triangulation lie outside the sphere. 
Each interior $(d-2)$-face $F$ is contained in exactly two $(d-1)$-faces,
$\conv(F\cup\{v_1\})$ and $\conv(F\cup\{v_2\})$; we call it \emph{locally Delaunay} if 
there exists a $(d-2)$-sphere passing through the vertices of $F$ 
such that the vertices $v_1$ and $v_2$ lie outside this sphere.

\begin{lemma}[Delaunay Lemma]\label{lemma:localDelaunay}
	Let $V\subset\R^{d-1}$ be a finite, affinely spanning set of points.
	A triangulation $\mathcal T$ of $\conv(V)$ with vertex set~$V$ 
	is the Delaunay triangulation if and only if one of the following equivalent 
	statements hold:
\begin{compactenum}[\rm(1)]
	\item All $(d-1)$-faces of $\mathcal T$ are {Delaunay}.
	\item All faces of $\mathcal T$ are {Delaunay}.
	\item All $(d-2)$-faces of $\mathcal T$ are {Delaunay}.
	\item All interior $(d-2)$-faces of $\mathcal T$ are {locally Delaunay}.
\end{compactenum}

\end{lemma}

\begin{proof}
 The first statement is the definition of a Delaunay triangulation. It implies 
 the second statement: For each face $F$ one can always do a slight change 
 to the supporting sphere of a $(d-1)$-face that contains $F$ to derive a supporting sphere of $F$.
 The second statement implies the third and this in turn the last one. 
 For more details and also for a proof that the last statement implies the first, 
 we refer to Edelsbrunner~\cite[pp.~7 and 99]{Edelsbrunner}.
\end{proof}

\begin{definition}[stellar subdivision]
	Let $p$ be a point inside a full-dimensional simplex $\sigma$ of a triangulation $\mathcal T$.
	A \emph{single stellar subdivision} of $\mathcal T$ at $\sigma$ (by $p$) is the triangulation that replaces
	$\sigma$ by the simplices spanned by $p$ and a proper face of $\sigma$.
	We call a triangulation a \emph{multiple stellar subdivision} of $\mathcal T$ at $\sigma$  
	if one or more single stellar subdivisions have been applied.  
\end{definition}

In the following, we will discuss which Delaunay triangulations 
can be generated by multiple stellar subdivisions of a triangulation that has just 
one full-dimensional simplex.  

\begin{lemma}\label{lem:stellar}
	A single stellar subdivision of the triangulation that has just one full-dimensional simplex
	is always a Delaunay triangulation.	 
\end{lemma}
		\begin{figure}[htb!]
			\centering	
		    \begin{tikzpicture}[scale=.7]
			% vertices
			\coordinate (A) at ($.8*(0,0)$);
			\coordinate (B) at ($.8*(7.7,.5)$);
			\coordinate (C) at ($.8*(3.9,6)$);
			\coordinate (Z) at ($.8*(4.1,2.6)$);
			% bounding box
			\clip[] (-.7,-.5) rectangle (7,5.5);
			% basics
			\draw[style=thick] (A) -- (B) -- (C) -- cycle; % triangle
			\draw[style=thick] (A)--(Z)--(B)--(Z)--(C); % inner edges
			% down circle 
			\coordinate (AB) at ($(A)!.5!(B)$);
			\coordinate (AC) at ($(A)!.5!(C)$);
			\coordinate (AZ) at ($(A)!.5!(Z)$);
			\coordinate (BZ) at ($(B)!.5!(Z)$);
			\coordinate (ABx) at ($(AB)!1cm!90:(B)$);
			\coordinate (ACx) at ($(AC)!1cm!90:(C)$);
			\coordinate (AZx) at ($(AZ)!1cm!90:(Z)$);
			\coordinate (BZx) at ($(BZ)!1cm!90:(Z)$);
			\coordinate (ABZ) at (intersection of AZ--AZx and BZ--BZx);
			\coordinate (ABC) at (intersection of AC--ACx and AB--ABx);
			\node (Circ) [draw,red] at (ABZ) [circle through={(Z)}] {};
			\node (Circ2) [draw] at (ABC) [circle through={(A)}] {};
			% line
			\draw[thick,red] (C) circle(.1);
			% nodes
			\draw[style=thick] 
			   +(3.2,1) node {$\sigma$} 
			   +(C) node [anchor=south] {$v$}
			   +(Z) node [anchor=north] {$c$};
		    \end{tikzpicture}
	    \caption{The circumsphere of $\sigma$ cannot contain $v$.}
	    \label{fig:stellar}
	\end{figure}
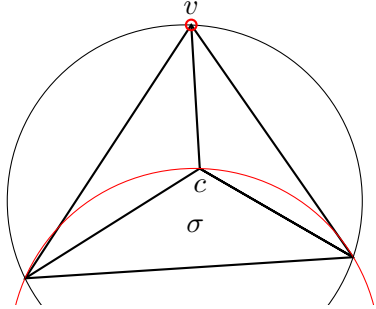

\begin{proof}
	If the circumsphere of a new full-dimensional simplex $\sigma$ would contain the vertex $v$ that does not lie in $\sigma$, 
	then it would contain all points of the original simplex, and hence it would contain the new vertex $c$ 
	in its interior. See Figure~\ref{fig:stellar}.
\end{proof}

\begin{definition}[dual tree of a stellar subdivision]\label{def:dualtree_triangulation}
	The \emph{rooted tree} $T_{\mathcal T}$ of a multiple stellar subdivision $\mathcal T$ of 
	a $(d-1)$-simplex $\sigma$, for $d\ge3$, 	has one node for every vertex that is inserted by a 
	single stellar subdivision, or equivalently for every $(d-1)$-face that it destroys.
	The root node $r$ corresponds to the (first) single stellar subdivision of $\sigma$.
	The node $v'$ is a child of the node $v$ if it corresponds to a single stellar subdivision
	of a $(d-1)$-face that was created in the single stellar subdivision corresponding to $v$. 
\end{definition}

Figure~\ref{Pic:DualTree2} shows a multiple stellar subdivision
of a $2$-simplex and the corresponding dual tree.

\begin{example}\label{example:simplex_subdivided}
Figure~\ref{Pic:LowerBound} illustrates the construction of a multiple
stellar subdivision of a $(d-1)$-simplex (generalizing Lemma~\ref{lem:stellar}) where the dual tree is a path.
\end{example}

  \begin{figure}[htb!]
      \centering
      \begin{tikzpicture}[scale=1]
      % bounding box
      \clip[] (-.5,-.5) rectangle (7,5.5);
      %points
      \coordinate (A) at (0,0);
      \coordinate (B) at (5,0);
      \coordinate (C) at (2.5,4);
      \coordinate (Z) at (2.1,1.5);
      \coordinate (X1) at ($(C)!.2!(Z)$);
      \coordinate (X2) at ($(C)!.4!(Z)$);
      \coordinate (X3) at ($(C)!.7!(Z)$);
      \coordinate (X4) at ($(C)!.9!(Z)$);
      \coordinate (X5) at ($(C)!1.1!(Z)$);
      \coordinate (X6) at ($(C)!1.2!(Z)$);
      \coordinate (X7) at ($(C)!1.3!(Z)$);
      \coordinate (X8) at ($(C)!1.4!(Z)$);
      % red faces
      \draw[fill,pink] (B)--(X2)--(X3)--cycle;
      \draw[fill,pink] (A)--(B)--(X5)--cycle;
      % lines
      \draw[style=thick] (A) -- (B) -- (C) -- cycle;
      \draw[style=thick] (C)--(X1)--(X2)--(X3)--(X4)--(X5);
      \draw[] (A)--(X1)--(B);
      \draw[] (A)--(X2)--(B);
      \draw[] (A)--(X3)--(B);
      \draw[] (A)--(X4)--(B);
      \draw[] (A)--(X5)--(B);
      % vertical line
      \draw[style=dashed] ($(C)!0!(Z)$)--($(C)!1.8!(Z)$);
      % circle down
      \coordinate (AX5) at ($(A)!.5!(X5)$);
      \coordinate (BX5) at ($(B)!.5!(X5)$);
      \coordinate (AX5x) at ($(AX5)!1cm!90:(X5)$);
      \coordinate (BX5x) at ($(BX5)!1cm!90:(X5)$);
      \coordinate (F) at (intersection of AX5--AX5x and BX5--BX5x);
      \node (Circ) [draw,red] at (F) [circle through={(X5)}] {};
      % circle side
      \coordinate (X23) at ($(X2)!.5!(X3)$);
      \coordinate (BX3) at ($(B)!.5!(X3)$);
      \coordinate (X23x) at ($(X23)!1cm!90:(X3)$);
      \coordinate (BX3x) at ($(BX3)!1cm!90:(B)$);
      \coordinate (F) at (intersection of X23--X23x and BX3--BX3x);
      \node (Circ) [draw,red] at (F) [circle through={(X3)}] {};
      % Points
      \draw[fill] (A) circle (2pt);
      \draw[fill] (B) circle (2pt);
      \draw[fill] (C) circle (2pt);
      \draw[fill] (X1) circle (2pt);
      \draw[fill] (X2) circle (2pt);
      \draw[fill] (X3) circle (2pt);
      \draw[fill] (X4) circle (2pt);
      \draw[fill] (X5) circle (2pt);
      \draw[] (X6) circle (2pt);
      \draw[] (X7) circle (2pt);
      \draw[] (X8) circle (2pt);
      
        \end{tikzpicture}
      \caption{Take a full dimensional simplex and a ray from a vertex to the interior of the simplex. 
           Apply $n$ single stellar subdivisions to the complex by introducing new vertices 
           on the ray. The result is a Delaunay triangulation.}
      \label{Pic:LowerBound}
  \end{figure}
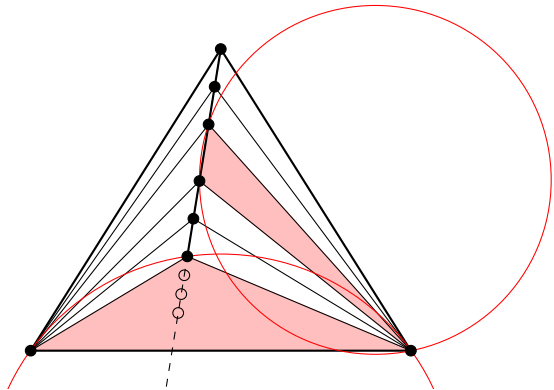

\begin{corollary}[a stellar subdivision can be undone]\label{cor:can_be_undone}
	Let the triangulation $\mathcal T'$
	be obtained from a triangulation $\mathcal T$
	of an affinely spanning point set $V\subset\R^{d-1}$ by a single stellar subdivision.
	If $\mathcal T'$ is a Delaunay triangulation,
	then so is $\mathcal T$.
\end{corollary}

\begin{proof}
	A stellar subdivision does not destroy a $(d-2)$-face, thus 
 	among the supporting spheres for $(d-2)$-faces in $\mathcal T'$ we have the supporting spheres for 
	all $(d-2)$-faces in $\mathcal T$.
\end{proof}

We end this section with a first example of a triangulation that
\emph{cannot} be realized as a Delaunay triangulation. 
The last sentence of the following lemma yields a more precise statement
that will turn out to be crucially important later.

\begin{lemma}\label{lemma:doublestellartriangle}
	Apply a single stellar subdivision by a point $x$ to a triangle $\conv\{A,B,C\}$    
	and then single stellar subdivisions by points $a,b,c$ to the three new triangles. 

	 The resulting triangulation is not a Delaunay triangulation:
	At least one of the three edges $Ax$, $Bx$ and $Cx$ violates the locally Delaunay criterion.
\end{lemma}
 
\begin{proof} 
	The nine angles at $a,b$ and $c$ sum up to $6\pi$. The three angles
	that lie in triangles that contain a boundary edge are each smaller than $\pi$,
	so the remaining six angles must sum to more than $3\pi$. 
	However, each of the three edges $Ax$, $Bx$ and $Cx$
	is locally Delaunay if and only if its two opposite angles sum to less than
	$\pi$. Hence not all three edges can be locally Delaunay. See Figure~\ref{Pic:Bound2D}.
\end{proof}
	\begin{figure}[htb!]
		\centering
	    \begin{tikzpicture}
		\coordinate (A) at (0,0);
		\coordinate (B) at (5,0);
		\coordinate (C) at (2.5,4);
		\coordinate (a) at (3.8,1.4);
		\coordinate (b) at (1.4,1.8);
		\coordinate (c) at (2.6,0.5);
		\coordinate (Z) at (2.5,1.5);
		\draw[style=thick] 
		   +(A) node [anchor=east] {$A$}
		   +(B) node [anchor=west] {$B$}
		   +(C) node [anchor=south] {$C$}
		   +(a) node [anchor=north east] {$a$}
		   +(b) node [anchor=north] {$b$}
		   +(c) node [anchor=south west] {$c$}
%		   +(2.4,2.4) node [anchor=south west] {$K$}
		   +(Z) node [anchor=south west] {$x$};
%		   +($(Z)!.4!(C)$) node [anchor=west] {$Cx$};
		\draw[style=thick] (A) -- (B) -- (C) -- cycle;
		\draw[style=thick] (A) --(Z)--(B)--(Z)--(C);
		\draw[] (A)--(c)--(B)--(a)--(C)--(b)-- cycle;
		\draw[ultra thick,red] (Z) -- (C);
		% circle less symetric:
		\coordinate (CZ) at ($(C)!.5!(Z)$);
		\coordinate (CZx) at ($(CZ)!3mm!90:(Z)$);
		\node (Circ) [draw,color=red,style=dashed] at (CZx) [circle through={(Z)}] {};
		\draw[color=red,thick]($(b)!.3cm!(Z)$) arc(5:71:.3cm);
		\draw[color=red,thick]($(a)!.3cm!(Z)$) arc(170:113:.3cm);
		\draw[color=red,thick]($(c)!.3cm!(A)$) arc(190:349:.3cm);
		\draw[] (a)--(Z)--(b)--(Z)--(c);
	      \end{tikzpicture}
		\caption{% 
		 In any triangulation of this combinatoral type 
		 at least one of the three edges $Ax$, $Bx$ and $Cx$ is not locally Delaunay.}
		\label{Pic:Bound2D}
	\end{figure}
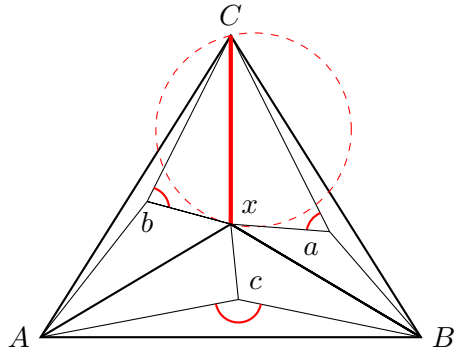

\subsection{Stereographic projection}\label{subsec:stereographic}

The \emph{stereographic projection} 
\[
\pi: S^{d-1}\setminus\{N\}\ \ \longrightarrow\ \ \R^{d-1}\times\{0\}
\]
is the bijective map that projects every point $x\neq N$ of the sphere $S^{d-1}$ 
along the ray through $x$ starting in
the north pole $N$ to the equator hyperplane of the sphere, which we identify with $\R^{d-1}$.
The inversion $x \mapsto N+2\frac{x-N}{\|x-N\|^2}$ in the sphere with center~$N$ 
and radius $\sqrt2$ extends this map to a bijection
$\widehat\pi:\R^d\cup\{\infty\}\rightarrow\R^d\cup\{\infty\}$. 
This sphere inversion is a Möbius transformation: It maps spheres to spheres,
where spheres through $N$ are mapped to hyperplanes (that is, spheres through $\infty$).

The stereographic projection identifies (the vertex sets of)
inscribed $d$-polytopes that have a vertex at the north pole
with (the vertex sets of) Delaunay subdivisions in $\R^{d-1}$.

\begin{proposition}[inscribed polytopes and Delaunay subdivisions]%
	\label{prop:Stereographic}% 
	Let $S^{d-1}$ denote the standard unit sphere in $\R^d$ and let $P$
	be an inscribed $d$-polytope whose vertex set $V\dotcup\{N\}\subset S^{d-1}$ 
	includes the north pole $N$ of $S^{d-1}$.
	
	Then $\pi(V)$ is the vertex set of a Delaunay subdivision in $\R^{d-1}$
	whose $(d-1)$-faces correspond to the facets of $P$ that do not contain
	the vertex $N$. 
	
	Conversely, if $W$ is a finite set that affinely spans $\R^{d-1}\times\{0\}$,
	then $\pi^{-1}(W)\cup\{N\}$ is the vertex set of an inscribed $d$-polytope 
	$P$ whose facets that miss $N$ are given by the $(d-1)$-faces of the Delaunay subdivision of~$W$
	and the facets of $P$ that contain the north pole $N$ are exactly the
	convex hulls $\conv(\pi^{-1}(F\cap W)\cup\{N\})$ given by the facet $F$ of $\overline P:=\conv(W)$.
\end{proposition}

In the following, we will need this theorem specialized to the simplicial case. 

\begin{proposition}[inscribed simplicial polytopes and Delaunay triangulations]%
	\label{prop:Stereographic_simplicial} 
	Let $S^{d-1}$ denote the standard unit sphere in $\R^d$  
	and let $P$ be an inscribed simplicial $d$-polytope whose vertex set 
	$V\dotcup\{N\}\subset S^{d-1}$ includes the north pole $N$ of $S^{d-1}$.
	
	Then $\pi(V)$ is the vertex set of a Delaunay triangulation in $\R^{d-1}$
	whose $(d-1)$-simplices correspond to the facets of $P$ that do not contain
	the vertex $N$. 
	
	Conversely, if $W\subset \R^{d-1}\times\{0\}$ is a finite set with a Delaunay triangulation $\mathcal T$
	whose convex hull $\overline P:=\conv(W)$ is a simplicial $(d-1)$-polytope 
	that has no points of $W$ on the boundary except for the vertices, 
	then $\pi^{-1}(W)\cup\{N\}$ is the vertex set of a simplicial inscribed $d$-polytope 
	$P=\conv(\pi^{-1}(W)\cup\{N\})$ whose
	facets that miss $N$ are given by the $(d-1)$-simplices of the Delaunay triangulation $\mathcal T$
	and the facets of $P$ that contain the north pole $N$ are exactly the
	convex hulls $\conv(\pi^{-1}(F\cap W)\cup\{N\})$ given by the facet $F$ of $\overline P$.
\end{proposition}

As a corollary to this we obtain that the 
Main Theorems \ref{mainthm:polytopes} (for polytopes) and~\ref{mainthm:Delaunay} (for Delaunay triangulations) are 
equivalent. We will prove Theorem~\ref{mainthm:Delaunay} below. We also obtain that the Lower Bound Theorem
is tight for inscribable polytopes, by applying Proposition~\ref{prop:Stereographic_simplicial} to Example~\ref{example:simplex_subdivided}:

\begin{corollary}\label{cor:LBTtight}
    For all $d\ge2$, $n\ge0$, there is an inscribed stacked $d$-polytope on $d+1+n$ vertices. 
\end{corollary}

We end this section with a proof of a very special (but crucial) part and case of Theorem~\ref{mainthm:Delaunay}.

\begin{proof}[Proof of the “only if” part of Theorem~\ref{mainthm:Delaunay} for $d-1=2$.]
	By Lemma \ref{lemma:localDelaunay} and Corollary \ref{cor:can_be_undone} 
	it suffices to show that the configuration of Figure~\ref{Pic:Bound2D}  
	cannot be realized as a Delaunay triangulation.
	(This was proved in Lemma~\ref{lemma:doublestellartriangle}.
	Via Proposition~\ref{prop:Stereographic_simplicial} it is equivalent to the fact
	that the polytope obtained by stacking onto all four facets of a tetrahedron is not
	inscribable, which was first proved by Steinitz \cite{Steinitz} \cite[Sect.~13.5]{Gruenbaum}.) 	
\end{proof}

\subsection{Some inscribable polytopes}\label{sec:f-inscribable}

\subsubsection{3-dimensional polytopes}\label{subsec:f-inscribable_3poly}

\begin{proposition}	
All $f$-vectors of $3$-polytopes occur for inscribable $3$-polytopes.
\end{proposition}

\begin{proof}
	According to Steinitz \cite{Steinitz} \cite[Sect.~10.3]{Gruenbaum}, the set of all $f$-vectors of convex $3$-polytopes
	is 
	\[
	\{(f_0,f_1,f_2)\in\Z^3: f_2\le 2f_0-4,\ f_0\le 2f_2-4,\ f_1=f_0+f_2-2\}. 
	\]
	In Figure~\ref{Pic:Wedge} one can see three types of inscribed $3$-polytopes.
	These can be described as wedges over an $n$-gon that have been stacked $k$ times. 
	By performing these constructions for arbitrary $n\ge3$ and $k\ge0$ we get three families of 
	inscribable $3$-polytopes.
	Their $f$-vectors are
	\begin{eqnarray*}
		f_{\textrm{left}}  &=&(2n-2+k,          3n-3+3k,n+1+2k)=(f_0,f_1,\tfrac{f_0}2+2+\tfrac32k  )\\
		f_{\textrm{middle}}&=&(2n-1+k,          3n-1+3k,n+2+2k)=(f_0,f_1,\tfrac{f_0}2+2+\tfrac32k+\tfrac12)\\
		f_{\textrm{right}} &=&(2n+k,\hspace{7mm}3n+1+3k,n+3+2k)=(f_0,f_1,\tfrac{f_0}2+2+\tfrac32k+1 )
	\end{eqnarray*}
			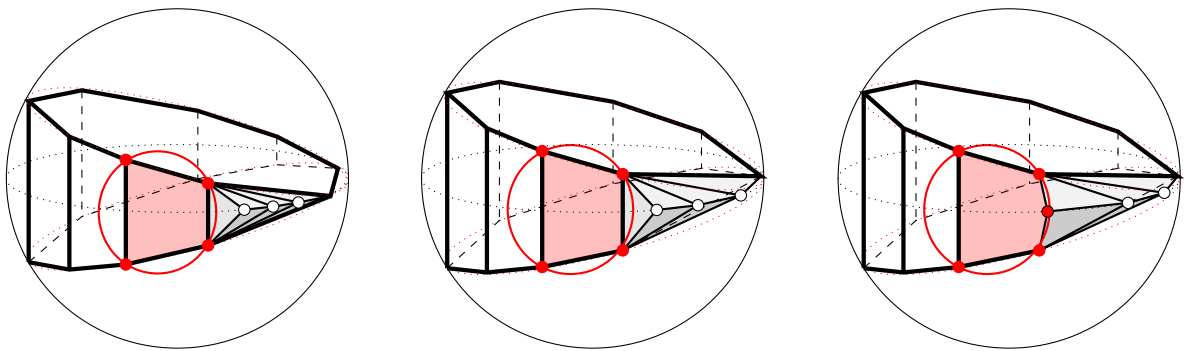
\begin{figure}[htb!]
				\centering	
			    \begin{tikzpicture}[scale=.75]
				\begin{scope}[yshift=-.16cm,xshift=3.1cm,rotate=-15,xshift=-3cm,yscale=.2,xscale=.97]
				\coordinate (Z) at (25:3cm);
				\coordinate (A) at (-25:3cm);
				\coordinate (B) at (280:3cm);
				\coordinate (C) at (250:3cm);
				\coordinate (D) at (225:3cm);
				\coordinate (E) at (180:3cm);
				\coordinate (F) at (135:3cm);
				\coordinate (G) at (90:3cm);
				\coordinate (H) at (60:3cm);
				\end{scope}

				\begin{scope}[yscale=.2]
				\coordinate (R1) at (-45:3cm);
				\coordinate (R2) at (-56:3cm);
				\coordinate (R3) at (-67:3cm);
				\end{scope}

				\coordinate (b) at ($(B)+(0,-1.1)$);
				\coordinate (c) at ($(C)+(0,-1.85)$);
				\coordinate (d) at ($(D)+(0,-2.35)$);
				\coordinate (e) at ($(E)+(0,-2.85)$);
				\coordinate (f) at ($(F)+(0,-2.23)$);
				\coordinate (g) at ($(G)+(0,-1.2)$);
				\coordinate (h) at ($(H)+(0,-.5)$);

				\draw[pink,fill] (B)--(C)--(c)--(b) ;
				\draw[] (0,0) circle (3);
		%	\draw[dotted] (0,0) ellipse (.6 and 3);
				\draw[dotted,red,xshift=3cm,rotate=15] (-2.9,.08) ellipse (2.9 and .6);
				\draw[fill,color=black!10] (B)--(R3)--(b);
				\draw[fill,color=black!5] (A)--(B)--(R3)--(R2)--(R1);
				\draw[fill,color=black!20] (b)--(R3)--(R2)--(R1)--(A) ;
				\draw[ultra thick] (A)--(b)--(c)--(d)--(e) (B)--(b) (C)--(c) (D)--(d) (E)--(e);
				\draw[ultra thick] (Z)--(A)--(B)--(C)--(D)--(E)--(F)--(G)--(H)--cycle;
				\draw[dashed] (e)--(f)--(g)--(h)--(Z) (F)--(f) (G)--(g) (H)--(h);
				\draw[thick] (A)--(R1)--(R2)--(R3) (B)--(R1)--(b) (B)--(R2)--(b) (B)--(R3)--(b);
				\draw[dotted,red,xshift=3cm,rotate=-15] (-2.9,-.1) ellipse (2.9 and .6);
				\draw[thick,red] (-.35,-.6) ellipse (1.03 and 1.08);
				\draw[dotted] (0,0) ellipse (3 and .6);
				\draw[draw=black,fill=white] +(R1) circle (.1) +(R2) circle (.1) +(R3) circle (.1);
				\draw[red,fill] +(B) circle(.1) +(C) circle(.1) +(c)circle(.1) +(b)circle(.1) ;
			    \end{tikzpicture}
			    \qquad%-------------------------------------------------------------------------------
			    \begin{tikzpicture}[scale=.75]
				\begin{scope}[xshift=3.1cm,rotate=-15,xshift=-3cm,yscale=.2,xscale=.95]
				\coordinate (A) at (0:3cm);
				\coordinate (B) at (280:3cm);
				\coordinate (C) at (250:3cm);
				\coordinate (D) at (225:3cm);
				\coordinate (E) at (180:3cm);
				\coordinate (F) at (135:3cm);
				\coordinate (G) at (90:3cm);
				\coordinate (H) at (55:3cm);
				\end{scope}

				\begin{scope}[yscale=.2]
				\coordinate (R1) at (-30:3cm);
				\coordinate (R2) at (-52:3cm);
				\coordinate (R3) at (-68:3cm);
				\end{scope}

				\coordinate (b) at ($(B)+(0,-1.35)$);
				\coordinate (c) at ($(C)+(0,-2.05)$);
				\coordinate (d) at ($(D)+(0,-2.55)$);
				\coordinate (e) at ($(E)+(0,-3.1)$);
				\coordinate (f) at ($(F)+(0,-2.45)$);
				\coordinate (g) at ($(G)+(0,-1.45)$);
				\coordinate (h) at ($(H)+(0,-.65)$);

				\draw[pink,fill] (B)--(C)--(c)--(b) ;
				\draw[] (0,0) circle (3);
	%			\draw[dotted] (0,0) ellipse (.6 and 3);
				\draw[dotted,red,xshift=3cm,rotate=15] (-2.9,0) ellipse (2.9 and .6);
				\draw[fill,color=black!10] (B)--(R3)--(b);
				\draw[fill,color=black!5] (A)--(B)--(R3)--(R2)--(R1);
				\draw[fill,color=black!20] (b)--(R3)--(R2)--(R1)--(b) ;
				\draw[ultra thick] (b)--(c)--(d)--(e) (B)--(b) (C)--(c) (D)--(d) (E)--(e);
				\draw[ultra thick] (A)--(B)--(C)--(D)--(E)--(F)--(G)--(H)--cycle;
				\draw[dashed] (e)--(f)--(g)--(h)--(A) (F)--(f) (G)--(g) (H)--(h);
				\draw[thick] (A)--(R1)--(R2)--(R3) (B)--(R1)--(b) (B)--(R2)--(b) (B)--(R3)--(b);
				\draw[dotted,red,xshift=3cm,rotate=-15] (-2.9,0) ellipse (2.9 and .6);
				\draw[thick,red] (-.39,-.55) ellipse (1.1 and 1.14);
				\draw[dotted] (0,0) ellipse (3 and .6);
				\draw[dashed] (A)--(b);
				\draw[draw=black,fill=white] +(R1) circle (.1) +(R2) circle (.1) +(R3) circle (.1);
				\draw[red,fill] +(B) circle(.1) +(C) circle(.1) +(c)circle(.1) +(b)circle(.1) ;
			    \end{tikzpicture}
			    \qquad%-------------------------------------------------------------------------------
			    \begin{tikzpicture}[scale=.75]
				\begin{scope}[xshift=3.1cm,rotate=-15,xshift=-3cm,yscale=.2,xscale=.95]
				\coordinate (A) at (0:3cm);
				\coordinate (B) at (280:3cm);
				\coordinate (C) at (250:3cm);
				\coordinate (D) at (225:3cm);
				\coordinate (E) at (180:3cm);
				\coordinate (F) at (135:3cm);
				\coordinate (G) at (90:3cm);
				\coordinate (H) at (55:3cm);
				\end{scope}

				\begin{scope}[yscale=.2]
				\coordinate (R1) at (-25:3cm);
				\coordinate (R2) at (-46:3cm);
				\coordinate (R3) at (-77:3cm);
				\end{scope}

				\coordinate (b) at ($(B)+(0,-1.35)$);
				\coordinate (c) at ($(C)+(0,-2.05)$);
				\coordinate (d) at ($(D)+(0,-2.55)$);
				\coordinate (e) at ($(E)+(0,-3.1)$);
				\coordinate (f) at ($(F)+(0,-2.45)$);
				\coordinate (g) at ($(G)+(0,-1.45)$);
				\coordinate (h) at ($(H)+(0,-.65)$);

				\draw[pink,fill] (B)--(C)--(c)--(b)--(R3) ;
				\draw[] (0,0) circle (3);
	%			\draw[dotted] (0,0) ellipse (.6 and 3);
				\draw[dotted,red,xshift=3cm,rotate=15] (-2.9,0) ellipse (2.9 and .6);
				\draw[fill,color=black!5] (A)--(B)--(R3)--(R2)--(R1);
				\draw[fill,color=black!20] (b)--(R3)--(R2)--(R1)--(b) ;
				\draw[ultra thick] (b)--(c)--(d)--(e)  (C)--(c) (D)--(d) (E)--(e);
				\draw[ultra thick] (A)--(B)--(C)--(D)--(E)--(F)--(G)--(H)--cycle;
				\draw[dashed] (e)--(f)--(g)--(h)--(A) (F)--(f) (G)--(g) (H)--(h);
				\draw[dashed] (A)--(b);
				\draw[thick] (A)--(R1)--(R2)--(R3) (B)--(R1)--(b) (B)--(R2)--(b) (B)--(R3)--(b);
				\draw[dotted,red,xshift=3cm,rotate=-15] (-2.9,0) ellipse (2.9 and .6);
				\draw[thick,red] (-.39,-.55) ellipse (1.1 and 1.14);
				\draw[dotted] (0,0) ellipse (3 and .6);
				\draw[draw=black,fill=white] +(R1) circle (.1) +(R2) circle (.1);
				\draw[red,fill] +(B) circle(.1) +(C) circle(.1) +(c)circle(.1) +(b)circle(.1) ;
				\draw[draw=black,fill=red](R3) circle (.1);
			    \end{tikzpicture}
			\caption{Three constructions for inscribed $3$-polytopes that produce all possible $f$-vectors.}
			\label{Pic:Wedge}
		\end{figure}
		
	\noindent	
	For $k=0$ the first type produces the $f$-vectors of simple polytopes;
	the first two types provide all $f$-vectors with 
	the minimal number of facets for any given number of vertices. 
	For $n=3$, the first type produces inscribable stacked $3$-polytopes with arbitrary number of vertices.
	These are simplicial and hence give the maximal number of facets for any given number of vertices.
	It is easy to see that for any number of vertices all permissible numbers of facets
	can be obtained by choosing the right $n$, $k$, and type.
\end{proof}
%%%%%%%%%%%%%%%%%%%%%%%%%%%%%%%%%%%%%%%%%%%%%%%%%%%%%%%%%%%%%%%%%%%%%%%%%%%%%%%%%%

\subsubsection{Neighborly polytopes}\label{subsec:f-inscribable_cyclic}

\begin{proposition}	
The $d$-dimensional cyclic polytope $C_d(n)$ with $n$ vertices
is inscribable for all $d\ge2$ and $n\ge d+1$.

Thus all $f$-vectors of neighborly polytopes occur for inscribable polytopes.
\end{proposition}

We will sketch three simple proofs for this.

\begin{proof}[Proof~1] 
    The  \emph{standard moment curve} in $\R^{d-1}$ is given by
	\[\gamma(t):=(t,t^2,\dots,t^{d-1}).\]
	This is a curve of order $d$ by Vandermonde's determinant formula.
	If the sequence of parameters $t_1<t_2<\dots< t_{n-1}$ grows fast enough, then for each $i>d$
	the point $\gamma(t_i)$ lies outside all circumspheres of the Delaunay triangulation of 
	$\{\gamma(t_1),\dots,\gamma(t_{i-1})\}$. Thus the Delaunay triangulation of $\{\gamma(t_1),\dots,\gamma(t_{i-1})\}$
	is obtained by induction on $i$, where for suitably large $t_i$ the new facets are given by the ``upper'' facets
	of $\conv\{\gamma(t_1),\dots,\gamma(t_{i-1})\}$ joined to the new vertex $\gamma(t_i)$.
	
	One checks, using Gale's evenness criterion, that thus the facets of the Delaunay triangulation
	of the point set $\{\gamma(t_1),\gamma(t_2),\dots,\gamma(t_{n-1})\}$  correspond exactly to the facets of $C_d(n)$
    that do not contain the last vertex.
    
    Finish the proof via Proposition~\ref{prop:Stereographic_simplicial}.
\end{proof}

\begin{proof}[Proof~2 {\rm(Seidel \cite[p.~521]{Seidel})}]
    The \emph{spherical moment curve} is given by
    \[C:\R^+\rightarrow\R^d, \qquad c(t):=\frac1{1+t^2+t^4+\dots+t^{2(d-1)}}(1,t,t^2,\dots,t^{d-1}). \]
    This curve lies on the image of the
    hyperplane $x_1=1$ under inversion in the origin, that is,
    on the sphere with center $\frac12e_1$ and radius $\frac12$.
    Using Descartes' rule of signs one gets that this curve
    (restricted to the domain $t>0$!) is of order $d$, and thus the convex hull of any
    $n$ distinct points on this curve is an inscribed realization of~$C_d(n)$.
\end{proof}

\begin{proof}[Proof~3 {\rm(Grünbaum \cite[p.~67]{Gruenbaum})}]
    For even $d\ge2$, we consider the \emph{trigonometric moment curve} 
	\[c:(-\pi,\pi]\rightarrow \R^{d}, \qquad c(t):=\left(\sin(t),\cos(t),~\sin(2t),\cos(2t),~\dots,~\sin(\tfrac d2t),\cos(\tfrac d2t)\right)\]
	Obviously its image lies on a sphere.
	We verify that this is a curve of order $d$ using the fact that any nonzero trigonometric polynomial of degree $\tfrac d2$ 
	has at most $d$ zeros per period (see e.g.\ Polya \& Szeg\H{o} \cite[pp.~72-73]{PolyaSzegoeII}). Thus we get 
	that the convex hull of any $n$ points on this curve yields an inscribed realization of $C_d(n)$.
	(Compare \cite[pp.~75-76]{Ziegler}.)
	
	For odd $d\ge3$, we check using Gale's evenness criterion that any ``vertex splitting'' on $C_{d-1}(n-1)$ results in a 
	realization of $C_d(n)$; this yields inscribed realizations of $C_d(n)$ where all vertices except for those labeled $1$ and $n$
	lie on a hyperplane. (See e.g.\ Seidel \cite[p.~528]{Seidel}, where the ``vertex splitting'' is called ``pseudo-bipyramid''.)
\end{proof}

\section{Stacked polytopes of dual degree at most 3 are inscribable}\label{sec:sufficiency}

The following proposition establishes the “if” part of 
Main Theorem \ref{mainthm:Delaunay} (and thus also of
Main Theorem \ref{mainthm:polytopes}).

\begin{proposition}\label{prop:2possible}
	Let $\mathcal T$ be a Delaunay triangulation in $\R^{d-1}$, 
	let $c$ be an interior vertex of degree $d$.
	Then one can perform single stellar subdivisions on two arbitrary $(d-1)$-faces $F_1$ and~$F_2$ 
	of $\mathcal T$ that contain $c$ such that the resulting triangulation
	is again Delaunay. 
\end{proposition}

\begin{proof}
	Let $F_1,\dots,F_d$ be the $(d-1)$-faces of $\mathcal T$ 
	that contain $c$, and let  $\mathcal R$
	 	be the set of all other $(d-1)$-faces of $\mathcal T$. Let 
	$v_1,\dots,v_d$ be the vertices of $F_1,\dots,F_d$ such that $v_i$ is
	 not contained in $F_i$. 
	The circumspheres of $F_3,\dots,F_d$ contain $c$.
	The intersection of the tangent hyperplanes to these $d-2$ spheres in the point $c$
	contains a line $t$ through $c$ that lies tangent to all those spheres.
	 
	Let $U$ be a small open ball around $c$ that, like $c$, lies outside the 
	circumspheres of all cells in $\mathcal R$.
	Then $U\cap t\setminus \{c\}$ consists of two disjoint, open  
	line segments. Choose two points, $x_1,x_2$, one in each line segment,
	and use them for single stellar subdivisions of $F_1$ resp.~$F_2$
	(cf.~Figure~\ref{Pic:Construct2}).

		\begin{figure}[htb!]
			\centering	
		    \begin{tikzpicture}
			% vertices
			\coordinate (A) at ($.8*(0,0)$);
			\coordinate (B) at ($.8*(7.7,.5)$);
			\coordinate (C) at ($.8*(3.9,6)$);
			\coordinate (Z) at ($.8*(4.1,2.6)$);
			% bounding box
			\clip[] (-.5,-.5) rectangle (7,5.2);
	%		\clip[draw] (4,1.8) circle (4.8cm);
			% basics
%			\fill[green!20!white] (A) -- (B) -- (C) -- cycle; % triangle green
			\fill[pink] (A) -- (B) -- (C) -- cycle; % triangle green
			\draw[style=thick] (A) -- (B) -- (C) -- cycle; % triangle
			\draw[white,style=fill] (Z) circle (1.1cm); % inner circle
			\draw[style=dotted] (Z) circle (1.1cm); % inner circle
			\draw[style=thick] (A)--(Z)--(B)--(Z)--(C); % inner edges
			% down circle 
			\coordinate (AZ) at ($(A)!.5!(Z)$);
			\coordinate (BZ) at ($(B)!.5!(Z)$);
			\coordinate (AZx) at ($(AZ)!1cm!90:(Z)$);
			\coordinate (BZx) at ($(BZ)!1cm!90:(Z)$);
			\coordinate (F) at (intersection of AZ--AZx and BZ--BZx);
			\node (Circ) [draw] at (F) [circle through={(Z)}] {};
			% new points
			\coordinate (V2) at ($(Z)!.9cm!90:(F)$);
			\coordinate (V1) at ($(Z)!-.9cm!90:(F)$);
			% tangente
			\draw[style=dashed] ($(V1)!-8cm!(V2)$)--($(V1)!8cm!(V2)$);
			\draw[style=fill] (V1) circle (.08cm);
			\draw[style=fill] (V2) circle (.08cm);
			% nodes
			\draw[style=thick] 
			   +(3.2,.5) node [anchor=south east] {$F_3$}
			   +(6.5,-.5) node [anchor=south east] {$C$}
			   +($(V1)!4cm!(V2)$) node [above] {$t$}
			   +(Z) node [below=9pt] {$U$}
			   +(A) node [anchor=south east] {$v_1$}
			   +(B) node [anchor=south west] {$v_2$}
			   +(C) node [anchor=south] {$v_3$}
			   +(V1) node [anchor=south west] {$x_1$}
			   +(V2) node [anchor=south east] {$x_2$}
			   +(Z) node [anchor=north] {$c$};
		    \end{tikzpicture}
				\caption{Choice of the subdivision points.}
				\label{Pic:Construct2}
	\end{figure}
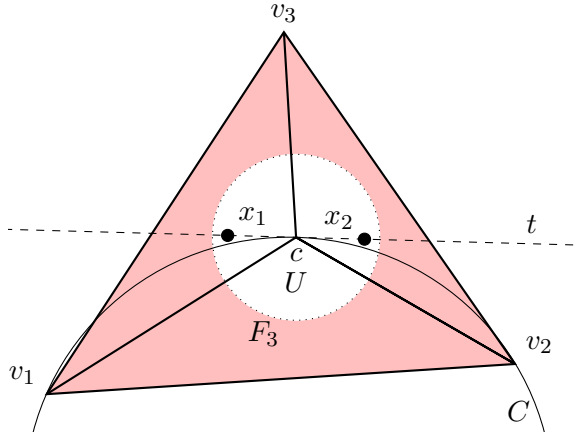

	We claim that the resulting triangulation $\mathcal T'$ is again Delaunay.  
	First we check that $x_1$ and $x_2$ lie inside $F_1$ resp.~$F_2$:
	They lie outside all facets $F_3,\dots,F_d$ but inside 
	$\conv\{v_1,\dots,v_d\}$, hence they lie in ${F_1}\cup{F_2}$. 
	Because $t$ contains $c$, which is a vertex of $F_1$ and $F_2$, only one component 
	of $t\setminus\{c\}$ can be contained in $F_1$ and only one can be contained in $F_2$.
	Hence we can assume that $x_1$ lies in the relative interior of $F_1$ 
	and $x_2$ lie in the relative interior of $F_2$.
	
	Now we need to show that all interior $(d-2)$-faces of $\mathcal T'$ 
	are locally Delaunay.
	The cells in $\mathcal R\cup\{F_3,\dots,F_d\}$ lie in both triangulations
	$\mathcal T$ and in $\mathcal T'$.
	They have empty circumspheres in $\mathcal T$ by assumption, and in $\mathcal T'$ by
	construction.
	
	Let $\mathcal I$ be the faces of $\mathcal T'$ that are not faces of~$\mathcal T$.
	It remains to show that all $(d-2)$-faces in $\mathcal T'$ that are
	contained in two facets of $\mathcal I$ are locally Delaunay.
	The first type lies in $(d-1)$-faces that both contain $x_1$ or both contain $x_2$.
	In this case the locally Delaunay condition is given by Lemma~\ref{lem:stellar}.
	The second type lies in a $(d-1)$-face that contains $x_1$ and
	an other $(d-1)$-face that contains $x_2$. There is only one such $(d-2)$-face, 
	namely the intersection of $F_1$ and $F_2$.  Let's call this face $K$.
	The circumsphere of $\conv(K\cup\{x_2\})$ does not contain $x_1$, because 
	$x_1,c,x_2$ are collinear and $c$ lies between $x_1$ and $x_2$
	(see Figure~\ref{Pic:Construct3}).
	Hence also $K$ is locally Delaunay and thus all interior $(d-2)$-faces of $\mathcal T'$ 
	are locally Delaunay. Hence $\mathcal T'$ is a Delaunay triangulation.
\end{proof}

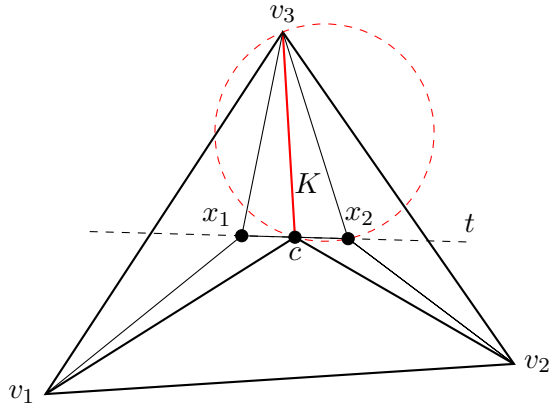
\begin{figure}[htb!]
	\centering
\begin{tikzpicture}
	% vertices
	\coordinate (A) at ($.8*(0,0)$);
	\coordinate (B) at ($.8*(7.7,.5)$);
	\coordinate (C) at ($.8*(3.9,6)$);
	\coordinate (Z) at ($.8*(4.1,2.6)$);
	% bounding sphere \clip[] (Z) circle (5.5cm);
	% basics
	\draw[style=thick] (A) -- (B) -- (C) -- cycle; % triangle
	\draw[style=thick] (A)--(Z)--(B); % inner edges
	\draw[style=thick,red] (Z)--(C); % inner edges
	% new points
	\coordinate (AZ) at ($(A)!.5!(Z)$);
	\coordinate (BZ) at ($(B)!.5!(Z)$);
	\coordinate (AZx) at ($(AZ)!1cm!90:(Z)$);
	\coordinate (BZx) at ($(BZ)!1cm!90:(Z)$);
	\coordinate (F) at (intersection of AZ--AZx and BZ--BZx);
	\coordinate (V2) at ($(Z)!.7cm!90:(F)$);
	\coordinate (V1) at ($(Z)!-.7cm!90:(F)$);
	% circumsphere
	\coordinate (CZ) at ($(C)!.5!(Z)$);
	\coordinate (V2Z) at ($(V2)!.5!(Z)$);
	\coordinate (CZx) at ($(CZ)!1cm!90:(Z)$);
	\coordinate (V2Zx) at ($(V2Z)!1cm!90:(Z)$);
	\coordinate (M) at (intersection of CZ--CZx and V2Z--V2Zx);
	\node [draw,red,dashed] at (M) [circle through={(Z)}] {};
	% tangent
	\draw[style=dashed] ($(V1)!-2cm!(V2)$)--($(V1)!3cm!(V2)$);
	% new edges
	\draw[] (A)--(V1)--(C)--(V2)--(B)--(V2)--(Z)--(V1);
	% nodes
	\draw[style=fill] (V1) circle (.08cm);
	\draw[style=fill] (V2) circle (.08cm);
	\draw[style=fill] (Z) circle (.08cm);
	\draw[] 
	   +($(V1)!3cm!(V2)$) node [anchor=south] {$t$}
	   +(A) node [anchor=east] {$v_1$}
	   +(B) node [anchor=west] {$v_2$}
	   +(C) node [anchor=south] {$v_3$}
	   +(V1) node [anchor=south east] {$x_1$}
	   +($(V2)+(.15,.3)$) node [] {$x_2$}
	   +(3.45,2.8) node [] {$K$}
	   +(Z) node [anchor=north] {$c$};
      \end{tikzpicture}
	\caption{The circumsphere does not contain $x_1$ because $x_1$, $c$ and $x_2$ are collinear.}
	\label{Pic:Construct3}
\end{figure}

Using this result, we also obtain examples of stacked polytopes that go beyond the 
rather special construction given by Corollary~\ref{cor:LBTtight}. 

\begin{corollary}[inscribable stacked polytopes with bounded vertex degree]\label{prop:BoundDeg}
   For all $d\ge2$ and $n\ge0$ there exists a stacked inscribed polytope of dimension $d$ that has
   $d+1+n$ vertices such that no vertex has degree more than~$2d$.
\end{corollary}

\begin{proof}
   We may assume $d>2$ (where the inductive steps discussed in the following do not destroy edges).
    
   We start with an arbitrary $d$-simplex, which is inscribed. All its vertices are simple; 
   we label them $1,\dots,d+1$. Now for $k=1,\dots,n$ we refer to Proposition~\ref{prop:2possible}
   in order to stack a new vertex $d+1+k$ onto the facet $\{k+1,k+2,\dots,k+d\}$
   at the simple vertex $d+1+(k-1)=d+k$. % $d+k,d+k-1,\dots,k+1$,
   This in particular destroys the facet $\{k+1,k+2,\dots,k+d\}$ 
   (which contains the vertex labeled $k+1$, which will not be touched again) and 
   creates the new facet $\{k+2,\dots,k+1+d\}$, adjacent to the new simple vertex $d+1+k$.
   
   In the stacked inscribed polytope created this way, vertices $i$ and $j$ are adjacent exactly if $|i-j| \le d$.
\end{proof}

\section{Three stellar subdivisions are impossible}\label{sec:necessity}

The following establishes the “only if” part of 
Main Theorem \ref{mainthm:Delaunay} (and thus also of
Main Theorem \ref{mainthm:polytopes}): If multiple stellar subdivisions are performed on three facets 
$F_1,F_2$ and $F_3$ at a simple interior vertex of an arbitrary triangulation,
then the resulting triangulation is not a Delaunay triangulation. 

For this, it suffices to consider the complex $\Delta$ that arises
by a single stellar subdivision of a $(d-1)$-simplex $\sigma\subset\R^{d-1}$ using 
an {arbitrary} interior point $c\in\sigma$. This complex $\Delta$
with $(d-1)$-faces $F_1,\dots,F_d$ is Delaunay by Lemma~\ref{lem:stellar}.
Now for $d\ge3$ we apply single stellar subdivisions to the cells $F_1,F_2,F_3$
by arbitrary interior points $r_1,r_2,r_3$. Our claim is that
the resulting triangulation $\mathcal T$ cannot be Delaunay.

In order to prove this claim, we first 
construct a point $x$ that depends only on~$\Delta$.
Its position with respect to $\Delta$ is established in Lemma~\ref{lemma:1}.
Then Lemma~\ref{lemma:2} records the properties of~$x$ with respect to
the subdivision $\mathcal T$. Finally, we establish in Proposition~\ref{prop:3impossible} that $\mathcal T$ cannot be 
Delaunay: For that we use an inversion in a sphere centered at $x$
in order to simplify the situation so that a projection argument reduces
the claim to the case $d=3$, which was established in Lemma~\ref{lemma:doublestellartriangle}.
\medskip

Let $\sigma=\conv\{v_1,\dots,v_d\}$ be a $(d-1)$-simplex in~$\R^{d-1}$,
let $c\in\sigma$ be an interior point, and let $\Delta$ be the 
single stellar subdivision of $\sigma$ by $c$, with $(d-1)$-faces
$F_1,\dots,F_d$, labeled such that $v_i\notin F_i$.

For some $k$ $(1\le k<d)$ let $\mathcal F:=\{F_{k+1},\dots,F_d\}$ and $\mathcal G:=\{F_1,\dots,F_k\}$.
Then $V_{\mathcal  F}:=\{v_{1},\dots,v_k\}$ is the set of vertices of~$\sigma$ that lie in all cells of~$\mathcal F$,
while $V_{\mathcal  G}:=\{v_{k+1},\dots,v_d\}$ is the set of vertices of~$\sigma$ that lie in all cells of~$\mathcal G$.

Now $E_{\mathcal F}:=\aff(V_{\mathcal F}\cup\{c\})$ is an affine subspace of dimension $k$, while
$E_{\mathcal G}:=\aff(V_{\mathcal G}\cup\{c\})$ has dimension~$d-k$. The two spaces together affinely span~$\R^{d-1}$,
so by dimension reasons they intersect in a line~$\ell$. This line intersects the two complementary faces
$\conv(V_{\mathcal F})$ and $\conv(V_{\mathcal G})$ of $\sigma$ in relatively interior points $\bar x$ resp.~$\bar y$. 
  
Let $C_{\mathcal F}$ denote the unique $(k-1)$-sphere that contains $V_{\mathcal F}\cup\{c\}$, that is, the
circumsphere of the $k$-simplex $\conv(V_{\mathcal F}\cup\{c\})$,
which is also the intersection of the circumspheres of $F_{k+1},\dots,F_d$.
The point $c$ lies in the intersection $\ell\cap C_{\mathcal F}$. The line $\ell$ also contains the point
$\ell\cap\conv(V_{\mathcal F})=\{\bar x\}$, which is a relative-interior point of $\conv(V_{\mathcal F})$ and thus
for $k>1$ lies in the interior of the circumspheres of $F_{k+1},\dots,F_d$
and thus in the interior of the sphere $C_{\mathcal F}$ relative to the subspace $E_{\mathcal F}$.
Thus $C_{\mathcal F}\cap\ell=\{c,x\}$, where the second intersection point $x$ is distinct from~$c$, and lies
outside $\sigma$ for $k>1$.

As for $C_{\mathcal F}$ and $x,\bar x$, we define $C_{\mathcal G}$ and $y,\bar y$ for ${\mathcal G}$:
See Figure~\ref{fig:lemma:1}.

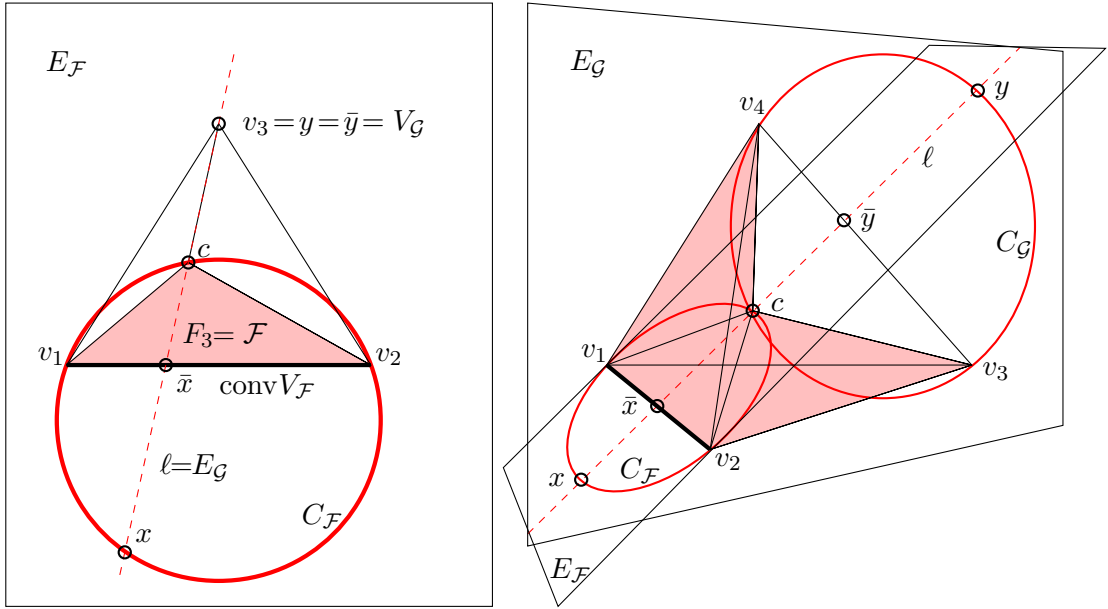
\begin{figure}[htb!]
	\centering
	    \begin{tikzpicture}[scale=.8]
		\draw  (-1,-4) rectangle (7,6);

		\coordinate (A) at (0,0);
		\coordinate (B) at (5,0);
		\coordinate (C) at (2.5,4);
		\coordinate (Z) at (2,1.7);
		\coordinate (r) at (3,2);
		\coordinate (AZ) at ($(A)!.5!(Z)$);
		\coordinate (BZ) at ($(B)!.5!(Z)$);
		\coordinate (CZ) at ($(C)!.5!(Z)$);

		\coordinate (AZx) at ($(AZ)!1cm!90:(Z)$);
		\coordinate (BZx) at ($(BZ)!1cm!90:(Z)$);
		\coordinate (CZx) at ($(CZ)!1cm!90:(Z)$);

		\coordinate (ABZ) at (intersection of AZ--AZx and BZ--BZx);
		\coordinate (BCZ) at (intersection of BZ--BZx and CZ--CZx);
		\coordinate (CAZ) at (intersection of AZ--AZx and CZ--CZx);

		\draw [fill,pink] (A) -- (B) -- (Z) -- cycle;

		\node (Circ) [draw,color=red,ultra thick] at (ABZ) [circle through={(Z)}] {};
%		\node [draw,color=red] at (BCZ) [circle through={(Z)}] {};
%		\node [draw,color=red] at (CAZ) [circle through={(Z)}] {};

		\coordinate (X) at (intersection of C--Z and Circ);
		\coordinate (Xbar) at (intersection of A--B and C--Z);
		\draw[style=thick] 
		   +(3.3,-.35) node {$\conv \!V_{\mathcal F}$}
		   +(4.2,-2.5) node  {$C_{\mathcal F}$}
		   +(2.1,-1.7) node  {$\ell\!\!=\!\!E_{\mathcal G}$}
		   +(-.25,.15) node {$v_1$}
		   +(5.3,.14) node  {$v_2$}
		   +($(C)+(1.9,0)$) node {$v_3\!=\!y\!=\!\bar{y}\!=V_{\mathcal G}$}
%		   +(2.9,2.1) node  {$F_1$}
%		   +(1.6,2) node  {$F_2$}
		   +(2.6,.5) node  {$F_3\!\!=\!\!\ {\mathcal F}$}
		   +(2.25,1.9) node {$c$}
		   +(0,5) node {$E_{\mathcal F}$}
		   +(X) node [anchor=south west] {$x$} 
		   +(Xbar) node [anchor=north west] {$\bar{x}$};
		\draw (A) -- (B) -- (C) -- cycle;
		\draw[ultra thick] (A) -- (B);
		\draw (A) --(Z)--(B)--(Z)--(C);
		\draw[dashed,color=red] ($(C)!-.5!(Z)$)--($(C)!3.3!(Z)$);
		\draw[style=thick] (X) circle (.1);
		\draw[style=thick] (Xbar) circle (.1);
		\draw[style=thick] (Z) circle (.1);
		\draw[style=thick] (C) circle (.1);
%		\draw[] (B)--(r)--(C)--(r)--(Z);
	      \end{tikzpicture}
	    \begin{tikzpicture}[scale=.8]
		\coordinate (A) at (0,0);
		\coordinate (B) at (6,0);
		\coordinate (C) at (2.5,4);
		\coordinate (D) at (1.7,-1.4);
		\coordinate (Z) at (2.4,.9);
		\coordinate (X) at (-.42,-1.9);
		\coordinate (Xbar) at (.83,-.68);
		\coordinate (Ybar) at (3.9,2.4);
		\coordinate (Y) at (6.1,4.55);
		% filled tetrahedron
		\draw [fill, pink] (A) -- (D) -- (B) -- (Z) -- (C) -- cycle;
		% circle
		\draw [thick,red,rotate=41] (.42,-1.08) ellipse (2cm and 1.1cm);
		\draw [thick,red] (4.54,2.3)  ellipse (2.5cm and 2.85cm);
		% line
		\draw [red,dashed] (-1.3,-2.8) --(6.8,5.27);
		% edges
		\draw (-1.3,-3) -- (-1.3,6) -- (7.5,5.2) -- (7.5,-1) -- cycle;
		\draw (-.8,-4) -- (-1.7,-1.7) -- (5.3,5.3) -- (8.2,5.25) -- cycle;
		\draw (A) -- (B) -- (C) -- cycle;
		\draw (A) --(D)--(B)--(D)--(C);
		\draw[ultra thick] (A) -- (D);
		\draw (A) --(Z)--(B)--(Z)--(C)--(Z)--(D);
		% marked nodes
		\draw [thick] (X) circle (.1);
		\draw [thick] (Y) circle (.1);
		\draw [thick] (Ybar) circle (.1);
		\draw [thick] (Xbar) circle (.1);
		\draw [thick] (Z) circle (.1);
		% labels
		\draw[style=thick] 
		   +($(A)+(-.2,.2)$) node  {$v_1$}
		   +($(B)+(.4,-.1)$) node  {$v_3$}
		   +($(C)+(-.1,.3)$) node  {$v_4$}
		   +($(D)+(.3,-.2)$) node  {$v_2$}
		   +($(Z)+(.4,.1)$) node {$c$}
		   +($(X)+(-.4,0)$) node {$x$}
		   +($(Xbar)+(-.45,0)$) node {$\bar{x}$}
		   +($(Ybar)+(.4,0)$) node {$\bar{y}$}
		   +($(Y)+(.4,0)$)  node {$y$}
		   +(-.3,5) node {$E_{\mathcal G}$}
		   +(.55,-1.75) node {$C_{\mathcal F}$}
		   +(6.7,2) node {$C_{\mathcal G}$}
%		   +(.15,-2.35) node {$C_{\mathcal F}$}
		   +(-.6,-3.4) node {$E_{\mathcal F}$}
%		   +(.4,-1.1) node {$\convQ$}
		   +(5,3.8) node [anchor=north west] {$\ell$};
	      \end{tikzpicture}
\caption{The situation of Lemma~\ref{lemma:1}. The left figure illustrates $d=3$, $k=2$, the right one $d=4$, $k=2$.}
\label{fig:lemma:1}
\end{figure}

\begin{lemma}\label{lemma:1}
In the situation just described, the point $x$ lies 
outside the circumspheres of $F_1,\dots,F_k$ and
on the circumspheres of $F_{k+1},\dots,F_d$.
\end{lemma}  

\begin{proof}
	The five points $x,\bar x, c,\bar y,y$ lie in this order along the line $\ell$,
	where the first two points coincide in the case $k=1$, while the last two coincide for~$d-k=1$.
	The circumspheres of $F_1,\dots,F_k$     intersect the line $\ell$ in $\{c,y\}$,
	and thus the point $x$ lies outside these spheres, while 
	the circumspheres of $F_{k+1},\dots,F_d$ intersect the line $\ell$ in $\{c,x\}$.
\end{proof}

\begin{lemma}\label{lemma:2}
If in the above situation the stellar subdivision of some or all of the facets $F_1,\dots,F_k$
results in a Delaunay triangulation $\mathcal T$, then the point $x$ lies outside all of the
circumspheres of the newly created $(d-1)$-faces.
\end{lemma}

\begin{proof} 
	Without loss of generality, let us assume that $\mathcal T$ is a single stellar subdivision 
	of $\Delta$ at $F_1$ by a new vertex $r$ inside $F_1$. This will result in $d$ new 
	facets $F'_1,\dots,F'_d$, whose vertex set consists of~$r$
	together with all-but-one of the vertices of $F_1$, which are $c,v_2,\dots,v_d$. 
	
    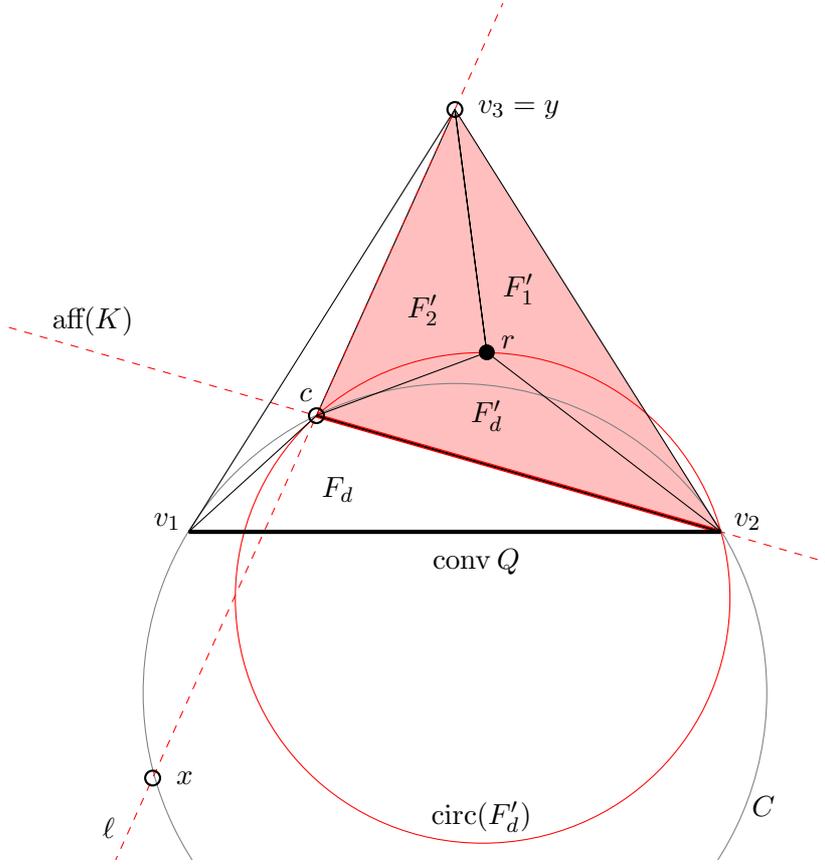
\begin{figure}[htb!]
    	\centering
    	    \begin{tikzpicture}[scale=1.4]
    		\clip [] (-1.7,-3.1) rectangle (6,5);

    		\coordinate (A) at (0,0);
    		\coordinate (B) at (5,0);
    		\coordinate (C) at (2.5,4);
    		\coordinate (Z) at (1.2,1.1);
    		\coordinate (r) at (2.8,1.7);
    		\coordinate (AZ) at ($(A)!.5!(Z)$);
    		\coordinate (BZ) at ($(B)!.5!(Z)$);
    		\coordinate (rZ) at ($(r)!.5!(Z)$);

    		\coordinate (AZx) at ($(AZ)!1cm!90:(Z)$);
    		\coordinate (BZx) at ($(BZ)!1cm!90:(Z)$);
    		\coordinate (rZx) at ($(rZ)!1cm!90:(Z)$);

    		\coordinate (ABZ) at (intersection of AZ--AZx and BZ--BZx);
    		\coordinate (rBZ) at (intersection of rZ--rZx and BZ--BZx);

    		\draw [fill,pink] (Z) -- (B) -- (C) -- cycle;
    		\draw [ultra thick,red] (B) -- (Z);
    		\draw[dashed,color=red] ($(B)!-1!(Z)$)--($(B)!3!(Z)$);

    		\node (Circ) [draw,color=gray] at (ABZ) [circle through={(Z)}] {};
    		\node [draw,color=red] at (rBZ) [circle through={(Z)}] {};

    		\coordinate (X) at (intersection of C--Z and Circ);
    		\coordinate (Xbar) at (intersection of A--B and C--Z);
    		\draw[style=thick] 
    		   +(2.7,-.28) node {$\conv Q$}
    		   +($(A)+(-.2,.1)$) node {$v_1$}
    		   +($(B)+(.25,.1)$) node {$v_2$}
    		   +($(C)+(.6,0)$) node {$v_3=y$}
    		   +($(r)+(.2,.1)$) node {$r$}

    		   +($(r)+(.3,.6)$) node {$F'_1$}
    		   +($(r)+(-.6,.4)$) node {$F'_2$}
    		   +($(r)+(0,-.6)$) node {$F'_d$}

    		   +(1.4,.4) node  {$F_d$}
    		   +(-.9,2) node  {aff$(K)$}
    		   +(5.4,-2.6) node  {$C$}
    		   +(2.75,-2.7) node  {circ($F'_d$)}
    		   +(-.75,-2.8) node  {$\ell$}

    		   +($(Z)+(-.1,.2)$) node {$c$}
    		   +($(X)+(.3,0)$) node {$x$};

    		\draw (A) -- (B) -- (C) -- cycle;
    		\draw[ultra thick] (A) -- (B);
    		\draw (A) --(Z)--(B)--(Z)--(C);
    		\draw[dashed,color=red] ($(C)!-.5!(Z)$)--($(C)!3.3!(Z)$);
    		\draw[style=thick] (X) circle (.07);
    		\draw[style=thick] (Z) circle (.07);
    		\draw[style=thick] (C) circle (.07);
    		\draw[fill] (r) circle (.07);
    		\draw[] (B)--(r)--(C)--(r)--(Z);
    	      \end{tikzpicture}
    \caption{The three cases of Lemma~\ref{lemma:2}, for $d=3$, $k=2$.}
    \label{fig:lemma:2}
    \end{figure}
    
	We discuss them in three different cases (see Figure~\ref{fig:lemma:2}):
	\begin{compactenum}[(I)]
		\item One new facet, say $F'_1$, does not contain $c$.
		Then $c$ lies outside the circumsphere of $F'_1$, while all the 
		vertices in $V_{\mathcal  G}:=\{v_{k+1},\dots,v_d\}$ are vertices
		of $F'_1$, so $\bar y$ lies inside the circumsphere, or on its boundary (in the case $d-k=1$).
		In either case we conclude that $x$ lies outside the circumsphere from the ordering on the
		line $\ell$ described in the proof of Lemma~\ref{lemma:1}.
		\item $k-1$ new facets $F'_2,\dots,F'_k$ do not contain a vertex $v_j$, $2\le j\le k$.
		In this case we argue as in Case~(I). 
		\item $d-k$ new facets $F'_{k+1},\dots,F'_d$ miss a vertex $v_j$, $(k+1\le j\le d)$ from~$V_{\mathcal  G}$. 
		Then $F'_j$ is adjacent to the facet $F_j$ because both share the 
		$(d-2)$-face $K:=\conv(\{c,v_2,\dots,v_d\}{\setminus}v_j)$. Their circumspheres intersect in $\aff(K)$.
		The line $\ell$ intersects $\aff(K)$ in $c$, hence $x$ and $\bar x$ lie on the same side of $\aff(K)$,
		as well as $v_1$, because $\bar x$ is a convex combination of $v_1$ and $\aff(K)$.
		So, the circumsphere of $F_j$ passes through $x$ and $v_1$ on the same side of $\aff(K)$.
		Because $\mathcal T$ is Delaunay, the circumsphere of $F'_j$ does not contain $v_1$ and hence 
		also not $x$.\vskip-9.5mm\mbox{}
	\end{compactenum}
\end{proof} 

\begin{proposition}\label{prop:3impossible} 
		Let $\Delta$ be a single stellar subdivision of a $(d-1)$-simplex 
		$\sigma=\conv\{v_1,\dots,v_d\}$ in $\R^{d-1}$ by an interior point $c\in\sigma$,
		so the facets of $\Delta$ are $F_i=\conv(\{c,v_1,\dots,v_d\}{\setminus}v_j)$.
		
		Let $\mathcal T$ arise from this Delaunay triangulation $\Delta$ by single 
		stellar subdivisions of $F_1,\dots,F_k$ by interior points $r_i\in F_i$ $(1\le i\le k)$.
		
		If $\mathcal T$ is a Delaunay triangulation, then $k<3$.
\end{proposition}

\begin{proof}
	For $d=3$ this was established in Lemma~\ref{lemma:doublestellartriangle}, so we assume 
	$d>3$. As a single stellar subdivision can be undone without destroying the Delaunay property 
	(Corollary~\ref{cor:can_be_undone}),
	it is enough to show that $\mathcal T$ cannot be a Delaunay triangulation if $k=3$.
	
	For the sake of contradiction we assume that such a $\mathcal T$ is a Delaunay triangulation.
	Then we are in the situation discussed above, where we find the point~$x$ on the line $\ell$,
	which by Lemmas~\ref{lemma:1} and~\ref{lemma:2}
	lies on the circumspheres of the facets $F_4,\dots,F_d$, 
	but outside the circumspheres of all other facets of~$\mathcal T$. 
	Let $\mathcal R$ denote the set of these other facets.
	The inversion of $\R^{d-1}$ in the unit sphere centered at $x$ sends all the vertices 
	of~$\mathcal T$ to new points in~$\R^{d-1}$.
	This inversion induces a simplicial map to a new triangulation~$\mathcal T'$.
	\[
	 \Psi:\mathcal R\ \ \longrightarrow\ \ \mathcal T'.
	\]
	As an abbreviation, we denote the images of $\Psi$ by a prime $()'$; for example, $\Psi(v_1)=v'_1$.
	Note that if we apply this to the images of simplices $\sigma,F_1,\dots,F_3$, then
	we refer to the simplices obtained by applying $\Psi$ to the vertices.

	The simplicial complex $\mathcal T'$ is a part of the unique Delaunay subdivision of its vertex set,
	because for all cells in $\mathcal R$, an empty circumsphere is mapped to an empty circumsphere.
	This in particular shows that $\mathcal T'$ is a simplicial complex.
	Let $r'_1,r'_2,r'_3$ be the three images of the vertices that where used to perform single
	stellar subdivisions to $F_1,F_2,F_3$.
	Then these vertices are also interior vertices of $\mathcal T'$ and hence $\mathcal T'$
	is the result of single stellar subdivisions of $F'_1,F'_2$ and $F'_3$ by $r'_1,r'_2,r'_3$. 
	The inversion centered at $x$ also implies that $c',v'_1,v'_2$ and $v'_3$
	lie in a commen $2$-plane, as their preimages lie on a $2$-sphere that passes through $x$.
	Note that no three of these four vertices lie in line. By checking
	some vertex incidences, we figure out that the structure of $\mathcal T'$ can be described as follows:
	Take a $(d-1)$-simplex, split it into three simplices by inserting a vertex $c'$
	in the interior of a $2$-face and then apply single stellar subdivisions to each of those three simplices
	by points $r'_1,r'_2$ and $r'_3$. In particular, the support of $\mathcal T'$ is convex, 
	so $T$ is the Delaunay triangulation of the set $\{c',r'_1,r'_2,r'_3,v'_1,\dots,v'_d\}$
	of $d+4$ points in $\R^{d-1}$. (See the left part of Figure~\ref{fig:Rd}.)

\begin{figure}[htb!]
	\centering
	    \begin{tikzpicture}[scale=.8,rounded corners=.66pt]
		\coordinate (A) at (0,0);
		\coordinate (B) at (6,0);
		\coordinate (C) at (2.5,4);
		\coordinate (D) at (1.5,-1.4);
		\coordinate (Z) at (2.7,-.5);
		\coordinate (r1) at (1.9,1.6);
		\coordinate (r2) at (2.75,1.7);
		\coordinate (r3) at (3.2,1.6);
		% filled face
		\draw [fill,pink] ($(Z)+(-.12,+.2)$) --(C)--(B)--cycle;
		\draw [red,dashed] (4.3,1.85) circle(2.76);
		\draw [red,dotted] (4.3,1.85) ellipse(2.76 and 1.25);
		\draw [fill, pink,rotate=0] ($(Z)!.5!(B)$) ellipse (1.76cm and .8cm);
		% ebene
		\draw (-2.3,-2) -- (-.5,1) -- (6.5,1) -- (8.3,-2) -- cycle;
		% edges
		\draw [ultra thick] (A) -- (D) -- (B) -- cycle;
		\draw [ultra thick] (A) --(Z)--(D);
		\draw [draw=red,rotate=0] ($(Z)!.5!(B)$) ellipse (1.76cm and .8cm);
		\draw (C) -- (r2) -- (B) -- (r2) -- (A) -- (r2) -- (Z);
		\draw [ultra thick,red] (Z) --(C)--(B)--cycle;
		\draw [ultra thick] (C) -- (A) -- (C) -- (D);
		% subdivisions
		\draw (C) -- (r1) -- (A) -- (r1) -- (D) -- (r1) -- (Z);
		\draw (C) -- (r3) -- (D) -- (r3) -- (B) -- (r3) -- (Z);
		% nodes 
		\draw[fill=white] (r1) circle (.1);
		\draw[fill=pink] (r2) circle (.1);
		\draw[fill=white] (r3) circle (.1);
		% labels
		\draw[style=thick] 
		   +($(A)+(-.3,0)$) node  {$v'_1$}
		   +($(D)+(-.4,-.1)$) node  {$v'_2$}
		   +($(B)+(.2,.3)$) node  {$v'_3$}
		   +($(C)+(0,.3)$) node  {$v'_4$}
		   +($(r1)+(-.3,.2)$) node {$r'_3$}
		   +($(r2)+(-.37,.2)$) node {$r'_2$}
		   +($(r3)+(.37,.2)$) node {$r'_1$}
		   +(7.5,-1.8) node [anchor=south] {$K$}
		   +(6.1,-1.2) node [anchor=south] {$C'$}
		   +(7.4,1.5) node [anchor=south] {$C$}
		   +($(Z)+(.1,-.2)$) node {$c'$};
	      \end{tikzpicture}
	    \begin{tikzpicture}
		\coordinate (A) at ($.7*(0,0)$);
		\coordinate (B) at ($.7*(8,0)$);
		\coordinate (C) at ($.7*(4.1,6)$);
		\coordinate (a) at ($.7*(5.2,2.1)$);
		\coordinate (b) at ($.7*(3,2.1)$);
		\coordinate (c) at ($.7*(4,1)$);
		\coordinate (Z) at ($.7*(4.1,2)$);
		\draw[style=thick] 
		   +($.7*(6.2,5.3)$) node {$C'$}
		   +(A) node [anchor=east] {$v'_1$}
		   +(B) node [anchor=west] {$v'_2$}
		   +(C) node [anchor=south] {$v'_3$}
		   +($(a)+(.55,-.2)$) node {$\pi(r'_1)$}
		   +($(b)+(-.55,-.2)$) node {$\pi(r'_2)$}
		   +($(c)+(.1,-.4)$) node {$\pi(r'_3)$}
		   +($(Z)+(-.2,.3)$) node {$c'$};
		\draw[style=thick] (A) -- (B) -- (C) -- cycle;
		\draw[style=thick] (A)--(Z)--(B);
		\draw[ultra thick,color=red] (Z)--(C);
		\draw[dotted] (A)--(c)--(B)--(a)--(C)--(b)-- cycle;
		\draw[dotted] (a)--(Z)--(b)--(Z)--(c);
		\node[draw,color=red]  at ($(C)!.5!(Z)$) [circle through={(Z)}] {};
		\draw (a)[fill] circle (.1);
		\draw (b)[fill] circle (.1);
		\draw (c)[fill] circle (.1);  
	      \end{tikzpicture}
  \caption{Left: An example for $d=4$. Right: The projection image in the $2$-plane $K$.} 
\label{fig:Rd}  
\end{figure}
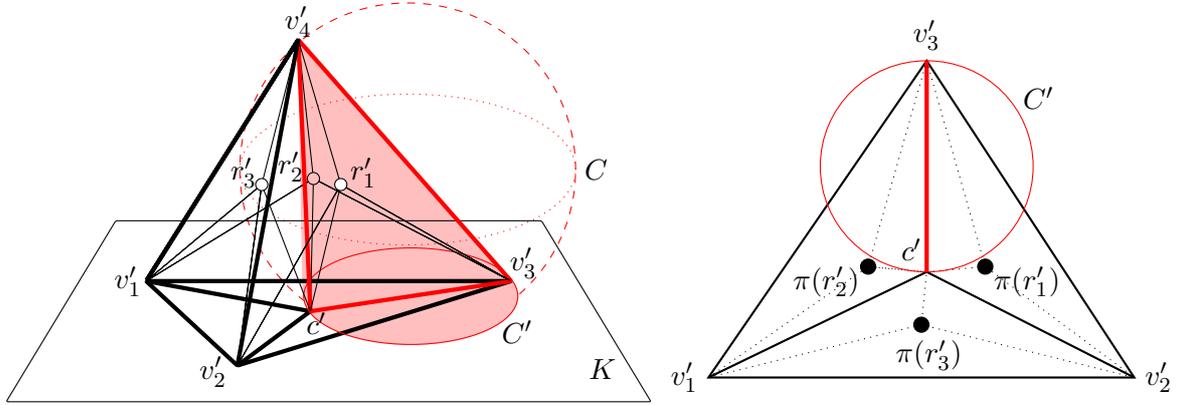  

	Let $K$ be the $2$-plane containing $c',v'_1,v'_2$ and $v'_3$ and let $\mathcal T'_K$ denote the 
	subcomplex of~$\mathcal T'$ that lies in $K$. We define barycentric coordinates by the points $v'_1,\dots,v'_d$
	and let $\pi$ be the corresponding coordinate projection
	\[
	\pi:\relint (\mathcal T')\rightarrow \relint (\mathcal T'_K)=\relint (\conv\{v_1,\dots,v_3\}).
	\]
	As $F_1$ and $F_2$ share a $(d-2)$-face, we know that $F'_1$ and $F'_2$
	also share the same $(d-2)$-face in~$\mathcal T'$.
	It has vertex set $\{c',v'_3,\dots,v'_d\}$ and it must have a supporting sphere.
	We pick one and call it~$C$. The intersection of~$C$ with~$K$ is a supporting sphere
	for the edge $(c',v'_3)$ in $\mathcal T'_K$. We call this $1$-sphere $C'$ and notice that
	the preimage $\conv(C')$ under 	$\pi$, which is contained in $\conv\! \left(C'\cup\{v'_4,\dots,v'_d\}\right)$, 
	lies completely inside $C$. 

	(This crucial fact is illustrated in red in the right part of Figure~\ref{fig:Rd}: The sphere that 
		contains $C'$ as well as $v'_4$ must enclose the whole truncated cone.)
	
	This implies that the images of $r'_1,r'_2$ and $r'_3$ under the projection~$\pi$
	lie outside $C'$, but in the interior of $\mathcal T'_K$. From this we derive that we 
	can apply single stellar subdivisions to the three $2$-faces of $\mathcal T'_K$ 
	by the vertices $\pi(r'_1),\pi(r'_2)$ and $\pi(r'_3)$ such that the edges $(c',v'_1),(c',v'_2)$ and $(c',v'_3)$ 
	would still be locally Delaunay, as indicated in the right part of Figure~\ref{fig:Rd}.
	However, in Lemma~\ref{lemma:doublestellartriangle} we have already proved that this is impossible.
\end{proof}

\begin{small}

\end{small}
% \bibliographystyle{siam}
% \bibliography{120Stacked} 

\end{document}